\documentclass[12pt]{article}
\usepackage{graphicx} 
\usepackage{enumitem}
\usepackage{xcolor}
\usepackage{latexsym,amsmath}
\usepackage{lscape}
\usepackage{color}
\usepackage[pagewise]{lineno}
\usepackage{graphicx} 
\usepackage {epstopdf}
\usepackage {subfigure}
\usepackage{hyperref}
\usepackage{booktabs} 
\usepackage{amsmath}
\usepackage{multicol}
\usepackage{multirow}
\usepackage{xcolor}
\usepackage{float}
\usepackage{amsthm}
\usepackage{amsfonts}
\usepackage[margin=1in]{geometry}
\usepackage[toc]{appendix}
\usepackage{stackengine}
\usepackage{bbold}
\RequirePackage[OT1]{fontenc}
\usepackage[authoryear]{natbib}
\usepackage{adjustbox}
\usepackage{amsmath,amssymb,amsthm,bm,latexsym}
\usepackage{graphicx,psfrag,epsf}
\usepackage[ruled,vlined]{algorithm2e}
\usepackage{epigraph,xcolor}
\usepackage{url}
\usepackage{booktabs}
\usepackage{multirow}
\usepackage{hyperref}
\usepackage{pifont}
\hypersetup{colorlinks=true,citecolor=blue,
	pdfpagemode=FullScreen,linktocpage=true}
\usepackage{cleveref}
\usepackage{amsfonts}
\usepackage{amsmath}
\usepackage{orcidlink}
\usepackage{amsthm}
\usepackage[T1]{fontenc}
\usepackage{multirow}
\usepackage{array}
\usepackage{booktabs}
\usepackage{accents}
\usepackage{cleveref}

\newtheorem{lemma}{Lemma}[section]

\newtheorem{theorem}{Theorem}[section]
\newtheorem{proposition}{Proposition}[section]
\newtheorem{assumption}{Assumption}[section]

\DeclareMathOperator{\Var}{Var}
\DeclareMathOperator{\Cov}{Cov}

\def\E{\mathbb{E}}
\def\Cov{{\textrm{Cov}}\,}
\def\R{\mathbb{R}}
\def\calN{\mathcal{N}}

\newcommand{\tran}{^\intercal}
\newcommand{\rd}{{\mathrm{d}}}
\newcommand{\EE}[2][]{\mathbb{E}_{#1}\left[#2\right]}
\newcommand{\ep}{\varepsilon}

\newcommand{\Indc}[1]{{\mathbf{1}\left\{{#1}\right\}}}

\newcommand{\RNum}[1]{\uppercase\expandafter{\romannumeral #1\relax}}

\newcommand{\norm}[1]{\left\| #1 \right\|}

\newcommand{\Rom}[1]{\uppercase\expandafter{\romannumeral #1\relax}}
\newcommand{\rom}[1]{\lowercase\expandafter{\romannumeral #1\relax}}

\newtheoremstyle{general}
{3mm} 
{3mm} 
{\it} 
{} 
{\bfseries} 
{.} 
{.5em} 
{} 

\theoremstyle{general}

\newtheorem{example}{Example}[section]

\newcommand{\eqd}{\stackrel{d}{=}}
\newcommand{\cv}{{\mathrm{CV}}}

\usepackage[most]{tcolorbox}

\definecolor{boxblue}{RGB}{235,242,250}
\definecolor{boxframe}{RGB}{70,105,145}

\newtcolorbox[auto counter, number within=section]{assumptionbox}[2][]{
  colback=blue!3,
  colframe=blue!50!black,
  fonttitle=\bfseries,
  coltitle=blue!50!black,
  colbacktitle=blue!8,
  title={Assumption~\thetcbcounter: #2},
  enhanced,
  breakable,
  attach boxed title to top left={xshift=0pt, yshift=-2pt},
  #1
}

\title{On the optimality of antithetic randomization for cross-validation}
\author{
  Srijan Chattopadhyay\textsuperscript{1}, 
  Sifan Liu\textsuperscript{2}, 
  Snigdha Panigrahi\textsuperscript{1}\thanks{The author gratefully acknowledges support from NSF CAREER Award DMS-2337882.} \\
  \vspace{0.2cm}\\
  \textsuperscript{1}Department of Statistics, University of Michigan, MI, USA \\
  \texttt{srijanch@umich.edu, psnigdha@umich.edu} \\
  \vspace{0.2cm}\\
  \textsuperscript{2}Department of Statistical Science, Duke University, NC, USA \\
  \texttt{sifan.liu@duke.edu}
}
\date{}

\begin{document}

\maketitle

\begin{abstract}
In the classical normal means problem, independent train--test folds can be constructed by perturbing the data with normal randomization. Averaging over $K$ such folds yields a cross-validation estimator whose bias depends on the marginal distribution of the randomization variables, while its variance depends on their joint distribution. This raises the questions: which joint law is optimal, and how to construct the corresponding randomization scheme? We show that: (i) for smooth estimators, antithetic randomization with pairwise correlation $\rho=-1/(K-1)$ is necessary and sufficient for the reducible variance due to randomization to remain bounded as the bias vanishes; (ii) a general construction yields a class of antithetic schemes, within which the jointly normal scheme is minimax optimal; and (iii) for non-smooth estimators with finitely many jump discontinuities, antithetic randomization improves the asymptotic rate of the reducible variance, while a simple control variate restores bounded variance when the discontinuities are known.
\end{abstract}

\textbf{Keywords:} antithetic sampling, risk estimation, prediction, variance reduction

\section{Introduction}
\label{sec: introduction}

Cross-validation is a standard method for estimating the test or prediction error of a fitted model. 
However, traditional cross-validation based on sample splitting typically relies on the assumption that the samples are independent and identically distributed (i.i.d.).
This assumption fails for clustered, longitudinal, or heterogeneous data, and does not apply in fixed-design problems, where sample splitting does not generate independent train--test folds.
Alternatively, one can split the information contained in the data rather than the samples themselves.
A broadly applicable implementation of this idea is feasible when the fitted model or estimator depends on the data through a sufficient statistic that is normal or asymptotically normal.
In this setting, independent train--test views of the data can be constructed by perturbing the statistic with normal randomization variables.
Repeating this construction $K$ times and averaging the resulting estimates yields an estimate of prediction error, directly paralleling traditional $K$-fold cross-validation.

Unlike traditional cross-validation, in which the chosen sample partition determines the dependence among folds, the construction of independent train--test views described above requires only randomization variables with normal marginals.
Obviously, these conditions do not uniquely determine the joint law of the randomization variables.
Importantly, as shown in this paper, the joint distribution of the randomization variables controls the reducible variance of the resulting cross-validation estimator and directly affects its performance.
This naturally raises the question of which joint law is optimal and how to construct the corresponding randomization scheme.

A classical strategy for variance reduction is to replace independent replicates with negatively correlated antithetic replicates, allowing their fluctuations to cancel when averaged. 
Antithetic sampling and quasi-Monte Carlo methods are standard implementations of this principle, and are applied in Monte Carlo simulation and generative modeling~\citep{hammersley1956new, owen2013monte, jia2026antithetic}.
Motivated by this strategy, \citet{liu2024cross} propose an antithetic randomization scheme for cross-validation, in which the $K$ randomization variables have pairwise correlation $-1/(K-1)$, and report substantially lower variance than the coupled bootstrap estimator based on independent randomization~\citep{oliveira2024unbiased}.

This paper asks whether antithetic randomization is optimal within the broader class of randomization schemes with normal marginals.
Because the same pairwise correlation can arise from different joint distributions, we seek to identify and construct the best-performing antithetic randomization scheme.
Our main contributions are summarized as follows:
\begin{enumerate}[leftmargin=*]
\item For smooth estimators, \Cref{thm:alpha to 0} in \Cref{sec: zero-sum} shows that antithetic randomization is not only sufficient but also necessary for the reducible variance of the cross-validation estimator to remain bounded as the perturbation level goes to zero, in which case the bias vanishes.
\item \Cref{lem: recipe:AR} in \Cref{sec: minimax} provides a general recipe for constructing antithetic randomization schemes, and \Cref{thm:minimax} establishes that the jointly normal antithetic scheme is the unique minimax-optimal construction within this class.
\item For non-smooth estimators with finitely many jump discontinuities, \Cref{thm:reducible-variance-nonsmooth} in \Cref{sec: nonsmooth} shows that antithetic randomization yields lower reducible variance than any non-antithetic scheme. Moreover, when the discontinuities are known in closed form, adding a suitable control variate restores the bounded reducible variance otherwise available only for smooth estimators.
\end{enumerate}

\section{Prediction error and randomization schemes}

\subsection{Problem setup and estimand}

We focus on the normal means problem as a representative setting; extensions to estimators depending on asymptotically normal sufficient statistics follow by the techniques of \citet{liu2024cross}.

Suppose we observe $Y\sim\calN(\mu,\sigma^2 I_n)$, where $\mu\in\R^n$ is the parameter of interest, and let $g:\R^n\to\R^n$ be an estimator of $\mu$.
Our goal is to estimate the prediction error
\begin{equation}
\label{eqn: pred:error}
\text{PE}(g)= \mathbb{E}\left[ \|g(Y) - \widetilde{Y}\|_2^2\right],
\end{equation}
where $\widetilde{Y}\sim \calN(\mu, \sigma^2 I_n)$ is an unobserved independent copy of $Y$, and the expectation is over both $Y$ and $\widetilde{Y}$.
When external predictors are present, $g(Y)=g_X(Y)$ with $X$ fixed; throughout, we omit $X$ for simplicity.
Since $\text{PE}(g)$ differs from the risk $\mathbb{E}[\|g(Y)-\mu\|_2^2]$ only by the constant $n\sigma^2$, the estimators considered here apply to either estimand.

\subsection{Cross-validation estimator}

The training error $\|g(Y)-Y\|_2^2$ is known to underestimate the prediction error in \eqref{eqn: pred:error}; but with a single observation, sample splitting is not feasible.
The idea instead is to split the information in $Y$ into two independent views, and to repeat the construction to obtain exchangeable train--test folds $(Y^k_{\text{train}}, Y^k_{\text{test}})$ for $k\in[K]=\{1,\ldots,K\}$.

For normal data with known $\sigma^2$, the two views are obtained by perturbing $Y$ with an independent normal randomization vector.
Specifically, let $\omega^k\sim\calN(0,\sigma^2I_n)$ be independent of $Y$, and fix $\alpha>0$; then the train--test folds
\begin{align}
\label{eqn: train-test}
    Y^k_{\text{train}} = Y + \sqrt\alpha\, \omega^k \quad \text{and}\quad  Y^k_{\text{test}} = Y - \frac{1}{\sqrt\alpha}\omega^k
\end{align}
are jointly normal with zero cross-covariance, and hence independent.
Both folds have the same mean as $Y$, but inflated variances, $(1+\alpha)\sigma^2$ and $(1+\alpha^{-1})\sigma^2$, where $\alpha$ controls how much information is allocated to training versus how much is reserved for testing.

Fitting $g$ on the train fold and evaluating it on the test fold, \citet{oliveira2024unbiased} estimate the prediction error by
\begin{equation}
\label{eqn: CVest}
\cv_{\alpha}= \frac{1}{K}\sum_{k=1}^K\Big\{ \big\|g(Y^{k}_{\text{train}}) - Y^{k}_{\text{test}}\big\|_2^2 -\frac{1}{\alpha}\|\omega^{k}\|_2^2 \Big\},
\end{equation}
where the first term averages predictive performance over the $K$ folds and the second corrects for the excess variance of $Y^k_{\text{test}}$ (relative to $\widetilde Y$); they show that its bias vanishes as $\alpha\downarrow0$.
For weakly differentiable $g$, the small-$\alpha$ limit of $\mathbb E(\cv_\alpha\mid Y)$ is
\[
\|Y-g(Y)\|_2^2+2\sigma^2\operatorname{tr}\{\nabla g(Y)\},
\]
which coincides with the Stein's unbiased risk estimate (SURE)~\citep{stein1981estimation} for prediction error. Thus, $\cv_\alpha$ can be viewed as a randomized analogue of SURE; see \citet{liu2024cross} for further discussion of this connection.

Each summand of \eqref{eqn: CVest} depends on the randomization only through $\omega^k$, so the expectation, and hence the bias, of $\cv_\alpha$ is determined by the common marginal law of the $\omega^k$ alone.
While the coupled bootstrap of \citet{oliveira2024unbiased} takes these randomization variables to be independent, $\cv_\alpha$ can be constructed from any scheme with normal marginals.
The construction in \eqref{eqn: train-test} is also known as data thinning or fission~\citep{neufeld2024data, leiner2025data}, and has been used in selective inference to split information between selection and inference~\citep{tian2018selective, rasines2023splitting, perry2026post}.
When $\sigma^2$ is unknown, it is replaced by a consistent estimator, as is standard in such splitting schemes; \citet{liu2024cross} demonstrate this plug-in approach for cross-validation.

\subsection{Normal randomization schemes}

The independence of $(Y^k_{\mathrm{train}},Y^k_{\mathrm{test}})$ requires only marginal normality of $\omega^k$, and places no constraint on the joint distribution of $(\omega^1,\ldots,\omega^K)$.
We therefore study the following class of schemes, assuming throughout that $(\omega^1,\ldots,\omega^K)$ is independent of $Y$ and that $K\geq2$.

\begin{assumption}[Exchangeable normal randomization]
\label{assump:normal-randomization}
  \begin{enumerate}
    \item[(i)]\label{assump:exch} \emph{Exchangeability:} \sloppy{$(\omega^{\pi(1)}, \ldots, \omega^{\pi(K)}) \eqd (\omega^1, \ldots, \omega^K)$} for every permutation $\pi$ of $[K]=\{1,\ldots,K\}$.
    \item[(ii)]\label{assump:marg} \emph{Marginal normality and equicorrelation:} $\omega^k\sim\calN(0,\sigma^2 I_n)$ for every $k\in[K]$, and $\Cov(\omega^k,\omega^{k'}) = \rho\sigma^2 I_n$ for every $k\neq k'$, where $\rho\in[-1/(K-1),1]$.
  \end{enumerate}
\end{assumption}

Part (i) ensures that, as in traditional cross-validation, the $K$ train--test pairs are exchangeable, whereas part (ii) captures the dependence between randomization variables through a single free parameter $\rho$.
The stated range of $\rho$ is exactly what makes the covariance matrix of $(\omega^1,\ldots,\omega^K)$ positive semidefinite.
At the lower endpoint $\rho=-1/(K-1)$, the randomization variables satisfy $\sum_{k=1}^K\omega^k=0$ almost surely; we call such a scheme, and the corresponding correlation, \emph{antithetic}.

Two existing estimators are special cases: the coupled bootstrap estimator of \citet{oliveira2024unbiased} takes $\rho=0$, and the cross-validation estimator of \citet{liu2024cross} takes $\rho=-1/(K-1)$.
The multi-fold thinning estimator of \citet{neufeld2024data}, although close in spirit, falls outside the class as its folds do not share the mean of $Y$.

Because the bias of $\cv_\alpha$ is determined by the common marginal law alone, every scheme satisfying Assumption~\ref{assump:normal-randomization} yields an estimator with the same bias; the schemes differ only in the variance that the randomization induces. \Cref{sec: zero-sum,sec: minimax,sec: nonsmooth} examine how this variance depends on the joint law of $(\omega^1,\ldots,\omega^K)$ in the small-bias regime $\alpha\downarrow0$, while $n$, $K$, and $\rho$ are held fixed.

\section{Necessity of antithetic randomization}
\label{sec: zero-sum}

We assess the estimator in~\eqref{eqn: CVest} through its mean squared error (MSE), which admits the decomposition:
\begin{align*}
\mathbb{E} \left[ \big(\cv_{\alpha} - \text{PE}(g) \big)^2 \right]
&= \underbrace{\left\{ \mathbb{E} [\cv_{\alpha}] - \text{PE}(g) \right\}^2}_{\text{Bias$^2$}} +  \underbrace{\Var \left[ \mathbb{E} [\cv_{\alpha} \mid Y] \right]}_{\text{S-VAR}} + \underbrace{\mathbb{E} \left[ \Var [\cv_{\alpha} \mid Y] \right]}_{\text{R-VAR}}.
\end{align*}
Here, $\mathrm{S\text{-}VAR}$ is the irreducible variance due to the sampling variance of $Y$, whereas $\mathrm{R\text{-}VAR}$ is the reducible variance contributed by the randomization variables $(\omega^1,\ldots,\omega^K)$.

Under Assumption~\ref{assump:normal-randomization}, both the bias and $\mathrm{S\text{-}VAR}$ depend only on $\alpha$ and on the common marginal distribution of the $\omega^k$, and the bias vanishes as $\alpha\downarrow0$ under mild integrability conditions~\citep{oliveira2024unbiased}.
By contrast, $\mathrm{R\text{-}VAR}$ depends on their joint distribution.
We therefore ask: for which dependence structures does $\mathrm{R\text{-}VAR}$ remain bounded as $\alpha \downarrow 0$, i.e., as the bias vanishes?
Within the class of Assumption~\ref{assump:normal-randomization}, \Cref{thm:alpha to 0} below answers this question completely: the reducible variance remains bounded if and only if the randomization is antithetic. The result requires the following regularity conditions on $g$.

\begin{assumption}[Weak differentiability and moments at an inflated variance]
\label{assump:weakdiff}
    All components $g_i$ ($1 \le i \le n$) of $g$ are weakly differentiable. That is, there exists a function $\nabla g_i : \mathbb{R}^n \to \mathbb{R}^n$, the weak derivative of $g_i$, such that
    $$
g_i(y + z) - g_i(y) = \int_{0}^{1} z \cdot \nabla g_i(y + tz)\mathrm{d}t,
$$
for almost all $y, z \in \mathbb{R}^n$. Denote the Jacobian matrix of $g$ as $\nabla g \in \mathbb{R}^{n \times n}$, where the $i$-th row is equal to $\nabla g_i$.
In addition, assume that there exists $\alpha_0>0$ such that, for
$Z_{\alpha_0}\sim\calN(\mu,(1+\alpha_0)\sigma^2I_n)$,
$\mathbb{E}\norm{g(Z_{\alpha_0})}_2^4 < \infty$,
$\mathbb{E}\norm{\nabla g(Z_{\alpha_0})}_{\mathrm{F}}^2 < \infty$.

\end{assumption}

\begin{theorem}[Necessity and sufficiency of antithetic randomization]
\label{thm:alpha to 0}
Suppose that $(\omega^1,\ldots,\omega^K)$ satisfy Assumption~\ref{assump:normal-randomization} for some fixed $\rho$, and that $g$ satisfies Assumption~\ref{assump:weakdiff} and \sloppy{$\EE{\|g(Y)-Y\|^2_2}>0$}.
Then, as $\alpha\downarrow 0$,
\[
\mathbb{E} \left[\Var \!\left(\cv_{\alpha} \mid Y\right)\right] =\begin{cases}
O(1) &\text{ if } \rho=-\frac{1}{K-1},\\
\Theta(\frac{1}{\alpha}) &\text{ if } \rho>-\frac{1}{K-1}.
\end{cases}
\]
\end{theorem}

The proof is given in Appendix~\ref{prf:thm:alpha to 0}, where the nondegeneracy condition $\E\norm{g(Y)-Y}_2^2>0$ is used only in the case $\rho>-1/(K-1)$.
Theorem~\ref{thm:alpha to 0} strengthens the results of \citet{liu2024cross} in two respects. First, it establishes bounded reducible variance under antithetic randomization without assuming pairwise or joint normality of the randomization variables. Second, it proves the converse: among all marginally normal schemes satisfying Assumption~\ref{assump:normal-randomization}, bounded reducible variance is attainable only at the antithetic correlation $\rho=-1/(K-1)$.

\section{Minimax optimal antithetic randomization}
\label{sec: minimax}

Section~\ref{sec: zero-sum} shows that, under Assumption~\ref{assump:normal-randomization}, antithetic correlation is the only correlation structure that keeps the reducible variance bounded.
This constrains the pairwise correlation of $(\omega^1,\ldots,\omega^K)$ but leaves their joint law free: many joint distributions with normal marginals satisfy the equivalent zero-sum constraint, and the jointly normal scheme of \citet{liu2024cross} is only one of them. We therefore ask whether some other antithetic scheme achieves smaller reducible variance.

We begin with a general recipe for constructing antithetic randomization schemes; the proof of the following proposition is deferred to Appendix~\ref{prf:lem:recipe:AR}.

\begin{proposition}[A general construction of antithetic randomization schemes]
\label{lem: recipe:AR}
Let $Z\sim\calN(0,\sigma^2I_d)$ with $d\geq n$, and let $M=(M^1,\ldots,M^K)$ be an exchangeable family of random matrices in $\R^{n\times d}$, independent of $Z$, satisfying
\begin{enumerate}[label=(\roman*)]
    \item\label{cond:co-isometry} \emph{co-isometry:} $M^k(M^k)\tran=I_n$ for every $k\in[K]$;
    \item\label{cond:sum-zero-matrices} \emph{sum-zero:} $\displaystyle\sum_{k=1}^K M^k=0$.
\end{enumerate}
Define $\omega^k=M^kZ$ for $k\in[K]$.
Then the stacked vector $(\omega^1,\ldots,\omega^K)\in\R^{nK}$ satisfies Assumption~\ref{assump:normal-randomization} with $\rho=-1/(K-1)$ or, equivalently, the zero-sum constraint $\sum_{k=1}^K\omega^k=0$.
\end{proposition}

The three examples below instantiate \Cref{lem: recipe:AR} with different exchangeable families $M$; the first recovers the jointly normal antithetic construction.

\begin{example}[Jointly normal construction]\label{ex:normal}
Let $Z^1,\ldots,Z^K\stackrel{\mathrm{iid}}{\sim}\calN(0,\sigma^2I_n)$ with average $\bar Z=K^{-1}\sum_{j=1}^K Z^j$, and let $Z\in\R^{nK}$ be the stacked vector of $Z^1,\ldots,Z^K$, so that $Z\sim\calN(0,\sigma^2I_{nK})$ and $d=nK$.
Let $M^k\in\R^{n\times d}$ be the scaled centering matrix determined by
$M^kZ=\{K/(K-1)\}^{1/2}(Z^k-\bar Z)$.
Applying a uniform random permutation to $(M^1,\ldots,M^K)$ makes the family exchangeable, and the conditions of \Cref{lem: recipe:AR} are then satisfied.
\end{example}

\begin{example}[Balanced Rademacher; $K$ even]\label{ex:rademacher}
Let $Z\sim\calN(0,\sigma^2I_n)$, and let $\eta$ be uniformly distributed over the balanced sign vectors $\{\eta\in\{\pm1\}^K:\sum_{k=1}^K\eta_k=0\}$, independently of $Z$.
Set $M^k=\eta_kI_n$, so that $\omega^k=\eta_kZ$.
This distribution is invariant under permutations of the coordinates of $\eta$, so $(M^1,\ldots,M^K)$ is exchangeable; each $M^k$ is orthogonal, and $\sum_{k=1}^K M^k=(\sum_{k=1}^K\eta_k)I_n=0$.
\end{example}

\begin{example}[Cyclic rotations; $K\ge2$, $n$ even]\label{ex:rotations}
Let $n=2m$ and $R=B^{\oplus m}\in O(n)$, where $B\in O(2)$ is the rotation through angle $2\pi/K$, so that $R^K=I_n$ and $\sum_{j=0}^{K-1}R^j=0$.
For $Z\sim\calN(0,\sigma^2I_n)$ and a uniform random permutation $\pi$ of $[K]$ independent of $Z$, set $M^k=R^{\pi(k)}$.
The permutation makes $(M^1,\ldots,M^K)$ exchangeable, while the orthogonality of $R$ and, since $R^K=I_n$, the identity $\sum_{k=1}^K M^k=\sum_{j=0}^{K-1}R^j=0$ verify the conditions of \Cref{lem: recipe:AR}.
\end{example}

The recipe of Proposition~\ref{lem: recipe:AR} therefore generates a rich class of antithetic schemes. The next result shows that the jointly normal construction of Example~\ref{ex:normal} is minimax optimal within this class; the proof is given in Appendix~\ref{prf:thm:minimax}.

\begin{theorem}[Minimax optimality of the jointly normal scheme]
\label{thm:minimax}
Let $\mathcal M$ denote the class of antithetic schemes constructed as per the procedure in Proposition~\ref{lem: recipe:AR}, and let
$\mathcal G=\left\{
g:\ g\text{ satisfies Assumption~\ref{assump:weakdiff} with }  \mathbb E\|\nabla g(Y)\|_{\mathrm F}^2=1\right\}$.
Then
\[
\inf_{M\in\mathcal M}\,
\sup_{g\in\mathcal G}\,
\lim_{\alpha\downarrow0}\,
\E \left[\operatorname{Var}(\cv_\alpha\mid Y)\right]=\frac{8\sigma^4}{K-1},
\]
and $M\in\mathcal M$ attains this value if and only if the stacked vector $(\omega^1,\ldots,\omega^K)$ is jointly normal.
\end{theorem}

\section{Extensions beyond smooth estimators}
\label{sec: nonsmooth}

\subsection{Variance for non-smooth estimators}

We extend the analysis to non-smooth $g$ with finitely many jump discontinuities. Specifically, we consider estimators of the form
\begin{equation}
\label{equ:nonsmooth:g}
g(y)
=
f_0(y)
+
\sum_{\ell=1}^m
a_\ell(y)\mathbf 1\{h_\ell(y)>0\}.
\end{equation}
Here, $f_0$ represents the smooth component of $g$, while each term $a_\ell(y)\mathbf{1}\{h_\ell(y)>0\}$ introduces a jump across the boundary $\{y:h_\ell(y)=0\}$.


\begin{assumption}[Piecewise-smooth predictors]
\label{ass:beyond smooth}
\begin{enumerate}[label=(\roman*)]
\item\label{ass:piecewise:regular} \textit{Regularity of the smooth components.}
The functions $f_0,a_1,\ldots,a_m$ satisfy the regularity conditions in Assumption~\ref{assump:weakdiff}.

\item \textit{Regularity of the thresholds.}
For each \(\ell=1,\ldots,m\), the function
\(h_\ell:\mathbb R^n\to\mathbb R\) is Lipschitz, with Lipschitz
constant \(L_\ell<\infty\).

\item \textit{Local boundary density and conditional jump moment.}
For each \(\ell\), there exists \(t_0>0\) such that \(h_\ell(Y)\) has a density
\(p_\ell\) on \((-t_0,t_0)\). Let
$
m_\ell(s)
=
\mathbb E\!\left[
\|a_\ell(Y)\|_2^2
\,\middle|\,
h_\ell(Y)=s
\right]$.
Assume that
\begin{equation}
\label{eq:bounded-jump}
\operatorname*{ess\,sup}_{|s|\le t_0}\; p_\ell(s) m_\ell(s)
<\infty.
\end{equation}
\end{enumerate}
\end{assumption}

Condition~\eqref{eq:bounded-jump} controls how often \(Y\) lies near a threshold boundary and the squared magnitude of the corresponding jump, thereby controlling the contribution of threshold-crossing events to the variance.

\begin{theorem}[Reducible variance for non-smooth predictions]
\label{thm:reducible-variance-nonsmooth}
Suppose that $(\omega^1,\ldots,\omega^K)$ satisfy Assumption~\ref{assump:normal-randomization}, that $g$ in \eqref{equ:nonsmooth:g} satisfies Assumption~\ref{ass:beyond smooth}, and that $\E[\|g(Y) - Y\|_2^2 ] > 0$.
Then, as $\alpha \downarrow 0$,
\[
\E\left[
\Var\{\cv_{\alpha}\mid Y\}
\right]= \begin{cases} 
O(\alpha^{-1/2}) & \text{if } \rho=-\frac{1}{K-1},\\[2pt]
\Theta(\alpha^{-1}) & \text{if } \rho>-\frac{1}{K-1}.
\end{cases}
\]
\end{theorem}

The proof is given in Appendix~\ref{prf:thm:reducible-variance-nonsmooth}. \Cref{thm:reducible-variance-nonsmooth} shows that, in the presence of jump discontinuities, the reducible variance of the antithetic scheme is no longer bounded. Nevertheless, it grows at rate \(O(\alpha^{-1/2})\), compared with \(\Theta(\alpha^{-1})\) for any non-antithetic scheme. This improvement requires no knowledge of the decomposition~\eqref{equ:nonsmooth:g}: computing \(\cv_\alpha\) requires only evaluations of \(g\), while the form~\eqref{equ:nonsmooth:g} is used solely in the analysis. \Cref{sec: numerical} confirms the two rates numerically for a ridge estimator whose coefficients are hard-thresholded to produce a sparse solution.

\subsection{Variance reduction by control variates}
\label{sec: control var}

The bound in \Cref{thm:reducible-variance-nonsmooth} still diverges as $\alpha\downarrow0$, due to the jump component of $g$. When the discontinuity boundaries of $g$ are known, this contribution can be removed by a control variate, restoring the bounded reducible variance of the smooth case.

For each $\ell\in[m]$, let
\[ 
r_\ell(\omega^k,Y) = \mathbb{E}\!\left[ \omega^k \mathbf{1} \bigl\{ h_\ell(Y+\sqrt{\alpha}\omega^k)>0 \bigr\} \,\middle|\, Y \right] - \omega^k \mathbf{1} \bigl\{ h_\ell(Y+\sqrt{\alpha}\omega^k)>0 \bigr\}.
\] 
Define the control variate
\begin{equation} \label{eqn: controlvariate} 
    \mathcal{V}_{\alpha} = \frac{2}{K\sqrt{\alpha}} \sum_{k=1}^K \sum_{\ell=1}^m \left\langle a_\ell(Y), r_\ell(\omega^k,Y) \right\rangle ,
\end{equation} 
and the adjusted estimator
\begin{equation*} 
    \widetilde{\cv}_{\alpha} = \cv_{\alpha} + \mathcal{V}_{\alpha}. 
\end{equation*} 
By construction, \(\mathbb E(\mathcal V_\alpha\mid Y)=0\). Hence
\(\widetilde{\cv}_\alpha\) and \(\cv_\alpha\) have the same conditional mean, and therefore the same bias and \(\mathrm{S\text{-}VAR}\); the adjustment affects only the reducible variance.

\begin{theorem}[Bounded reducible variance with control variates] 
\label{thm:controlvariate} 
Let $(\omega^1,\ldots,\omega^K)$ satisfy Assumption~\ref{assump:normal-randomization} for some fixed $\rho$. Let $g$ in~\eqref{equ:nonsmooth:g} satisfy the smooth-component conditions in Assumption~\ref{ass:piecewise:regular}, and assume that $ \mathbb{E}\left[ \|f_0(Y)-Y\|_2^2 \right] >0$. Then 
\[
\mathbb{E}\left[\Var\!\left(\widetilde{\cv}_{\alpha} \mid Y\right)\right] =O(1) \;\; \text{ as } \alpha \downarrow 0 \quad \text{ if and only if }\; \rho=-1/(K-1).
\]
\end{theorem}

The proof is given in Appendix~\ref{prf:thm:controlvariate}. Theorem~\ref{thm:controlvariate} shows that the control variate restores the bounded reducible variance of the smooth setting. Its implementation requires evaluating \[ \mathbb{E}\!\left[ \omega^k \mathbf{1} \bigl\{ h_\ell(Y+\sqrt{\alpha}\omega^k)>0 \bigr\} \,\middle|\, Y \right] \] for each threshold function. We now present an example in which this conditional expectation is available in closed form.

\noindent{\emph{A non-smooth ridge regression estimator with hard thresholding.}} \quad 
Let $X\in\mathbb R^{n\times p}$ be the fixed design matrix, and, for a regularization parameter $\lambda>0$, consider the ridge estimator
\begin{equation*}
\widehat\beta^{\mathrm{ridge}}_{\lambda}(y)= (X\tran X+\lambda I_p)^{-1}X\tran y.
\end{equation*}
To produce a sparse solution, we apply coordinatewise hard thresholding to the ridge estimator, yielding the non-smooth estimator
\begin{equation*}
\widehat\beta^{\mathrm{HT}}_{\lambda,\tau_j,j}(y)=\widehat\beta^{\mathrm{ridge}}_{\lambda,j}(y)
\Indc{
\left|\widehat\beta^{\mathrm{ridge}}_{\lambda,j}(y)\right|>\tau_j
},
\qquad j=1,\ldots,p,
\end{equation*}
where $\tau_j = c \times \mathrm{sd}(\widehat\beta^{\mathrm{ridge}}_{\lambda, j})$ for a fixed constant $c$, and $\mathrm{sd}(\widehat\beta^{\mathrm{ridge}}_{\lambda, j})$ is the standard deviation of the $j$-th coordinate of the ridge estimator.
Let $g^{\mathrm{HT}}_{\lambda}(y)=X\widehat\beta^{\mathrm{HT}}_{\lambda,\tau}(y)$.
This estimator belongs to the non-smooth class in
\eqref{equ:nonsmooth:g}; Appendix~\ref{app:hard-thresholded-ridge-control-variate} derives the the control variate~\eqref{eqn: controlvariate} analytically and \Cref{sec: numerical} presents numerical results for this example.

\section{Numerical results}
\label{sec: numerical}

We consider a sparse linear model with normal errors,
$Y_{n\times1}=X_{n\times p}\beta_{p\times1}+\ep_{n\times1}$, $\ep\sim\calN(0,\sigma^2I_n)$,
with $n=200$, $p=50$, and $s=10$ nonzero coefficients drawn from $\mathrm{Unif}(-1,1)$.
The design matrix $X$ has independent standard normal entries; $X$ and $\beta$ are drawn once and held fixed across all replicates.
The error variance is set to $\sigma^2=\mathrm{Var}(X\beta)/2$, corresponding to a signal-to-noise ratio of $2$.
We estimate the prediction error of the ridge and hard-thresholded ridge estimators of Section~\ref{sec: control var} using $\cv_\alpha$ in \eqref{eqn: CVest} with $K=6$ folds and $\lambda=10$, taking $\tau_j=1.65\,\mathrm{sd}(\widehat\beta^{\mathrm{ridge}}_{\lambda,j})$.
Each simulation is repeated 100 times independently, with variation indicated by the error bands in the plot.
Code to reproduce the numerical results is available at \texttt{github.com/liusf15/Antithetic-CV/tree/optimality}.

Panel (a) of \Cref{fig:plot1} plots the reducible variance of $\cv_\alpha$ against $1/\alpha$ for the ridge estimator under four schemes: independent randomization, $\rho=0$; jointly normal randomization with $\rho=1/(K-1)$; and the antithetic schemes of \Cref{ex:rademacher} and \Cref{ex:normal}, labelled ACV (Rademacher) and ACV (Normal).
The reducible variance grows linearly in $1/\alpha$ under the two non-antithetic schemes but remains bounded under both antithetic ones, as predicted by \Cref{thm:alpha to 0}. ACV (Normal) consistently achieves lower reducible variance than ACV (Rademacher), corroborating \Cref{thm:minimax}.
Panel (b) repeats the comparison for the hard-thresholded ridge estimator, replacing ACV (Rademacher) by ACV (Adjusted), which combines ACV (Normal) with the control variate of Section~\ref{sec: control var}.
The linear growth under the non-antithetic schemes persists; ACV (Normal) grows far more slowly, consistent with the $O(\alpha^{-1/2})$ bound of \Cref{thm:reducible-variance-nonsmooth}, and ACV (Adjusted) remains bounded, as guaranteed by \Cref{thm:controlvariate}.
Appendix~\ref{app:additional:figures} reports the corresponding MSE.
Panel (c) fixes $\alpha=0.01$ and plots the estimated prediction error of the hard-thresholded ridge estimator against $\lambda$, with the true prediction error shown as a dotted line.
All schemes share the same expectation and their mean curves lie close to the truth, but independent randomization produces substantially wider uncertainty bands than the antithetic schemes.

\begin{figure}[h!]
    \centering
    \includegraphics[width = .95\linewidth]{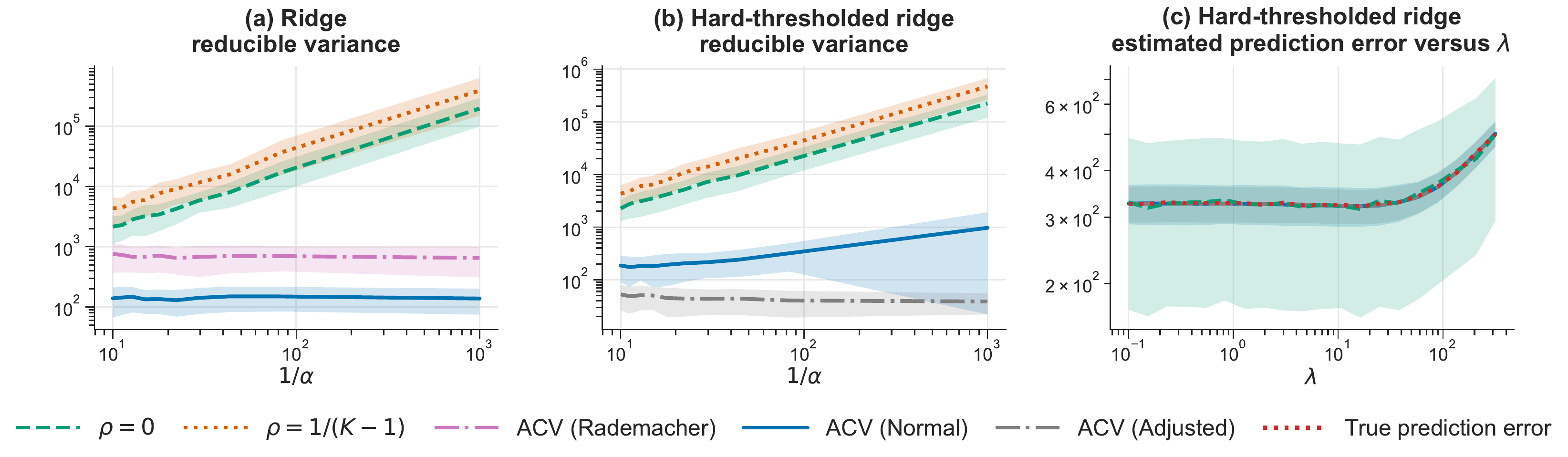}
    \caption{Reducible variance of $\cv_\alpha$ against $1/\alpha$ for ridge regression, panel (a), and for coordinatewise hard-thresholded ridge regression, panel (b); estimated and true prediction error of the hard-thresholded ridge estimator against $\lambda$, panel (c).
    }
    \label{fig:plot1}
\end{figure}

\bibliography{ref}

\appendix

\section{Proofs}
\subsection{Proof of Theorem~\ref{thm:alpha to 0}}
\label{prf:thm:alpha to 0}

\begin{lemma}[Small-$\alpha$ limit]
\label{lem:smooth-small-noise-limit}
Suppose that $g$ satisfies Assumption~\ref{assump:weakdiff}, and let
$\omega^1,\ldots,\omega^K$ be independent of $Y$, with
$\omega^k\sim\calN(0,\sigma^2I_n)$ marginally for every $k$. Define
\begin{align}
R_\alpha
&:=
\frac1K\sum_{k=1}^K
\left\{
\|Y-g(Y+\sqrt\alpha\,\omega^k)\|_2^2
+
2\int_0^1
(\omega^k)\tran
\nabla g(Y+t\sqrt\alpha\,\omega^k)
\omega^k\,\rd t
\right\},\label{equ:def:R_alpha} \\ 
R_0
&:=
\|Y-g(Y)\|_2^2
+
\frac2K\sum_{k=1}^K
(\omega^k)\tran\nabla g(Y)\omega^k.\notag
\end{align}
Then
\[
R_\alpha\longrightarrow R_0
\qquad\text{in }L_2
\quad\text{as }\alpha\downarrow0.
\]
Consequently,
\[
\operatorname{Var}(R_\alpha\mid Y)
\longrightarrow
\operatorname{Var}(R_0\mid Y)
\qquad\text{in }L_1.
\]
\end{lemma}

\begin{proof}[Proof of \Cref{lem:smooth-small-noise-limit}]
For each $k$, let
\[
D_{\alpha,k}
:=
g(Y+\sqrt\alpha\,\omega^k)-g(Y).
\]
Applying \Cref{lem:gaussian-shift-continuity} with
$f=g$, $\eta=\alpha_0$, and $(p,m)=(4,0)$ gives
\[
\|D_{\alpha,k}\|_{L_4}\longrightarrow0.
\]
Therefore,
\begin{align*}
\left\|
\|Y-g(Y+\sqrt\alpha\,\omega^k)\|_2^2
-
\|Y-g(Y)\|_2^2
\right\|_{L_2}&=\left\| \|Y - g(Y) - D_{\alpha,k}\|_2^2 - \|Y - g(Y)\|_2^2 \right\|_{L_2}\\
&= \left\|- 2 \langle Y-g(Y), D_{\alpha,k}\rangle + \|D_{\alpha,k}\|_2^2  \right\|_{L_2}\\
&\leq 2\left\|  \|Y-g(Y)\|_2\| D_{\alpha,k}\|_2 \right\|_{L_2} + \left\| \|D_{\alpha,k}\|_2^2 \right\|_{L_2}\\
&\leq 2\|Y-g(Y)\|_{L_4}\|D_{\alpha,k}\|_{L_4}
+
\|D_{\alpha,k}\|_{L_4}^2\\
&\longrightarrow0.
\end{align*}

Next, define
\[
E_{\alpha,k}
:=
2\int_0^1
(\omega^k)\tran
\left\{
\nabla g(Y+t\sqrt\alpha\,\omega^k)-\nabla g(Y)
\right\}
\omega^k\,\rd t.
\]
By Jensen's inequality and
$|w^\top A w|^2\leq\|w\|_2^4\|A\|_{\mathrm F}^2$,
\begin{align*}
\|E_{\alpha,k}\|_{L_2}^2
&\leq
4\int_0^1
\mathbb E\!\left[
\|\omega^k\|_2^4
\left\|
\nabla g(Y+t\sqrt\alpha\,\omega^k)-\nabla g(Y)
\right\|_{\mathrm F}^2
\right]\rd t\\
&\leq
4\sup_{0\leq r\leq\sqrt\alpha}
\mathbb E\!\left[
\|\omega^k\|_2^4
\left\|
\nabla g(Y+r\omega^k)-\nabla g(Y)
\right\|_{\mathrm F}^2
\right].
\end{align*}
The last expression converges to zero by
\Cref{lem:gaussian-shift-continuity}, applied with
$f=\nabla g$, $\eta=\alpha_0$, and $(p,m)=(2,4)$.

Moreover,
\[
\mathbb E\left[
\left|
(\omega^k)\tran\nabla g(Y)\omega^k
\right|^2
\right]
\leq
\mathbb E\|\omega^k\|_2^4\,
\mathbb E\|\nabla g(Y)\|_{\mathrm F}^2
<\infty,
\]
where we used the independence of $\omega^k$ and $Y$. Thus $R_0\in
L_2$. Since $K$ is fixed, applying triangle inequality yields
\[
\|R_\alpha-R_0\|_{L_2}\longrightarrow0.
\]

It remains to establish the conditional-variance conclusion. Let
\[
QX:=X-\mathbb E[X\mid Y],
\]
and thus
\[
\operatorname{Var}(X\mid Y)
=
\mathbb E[(QX)^2\mid Y].
\]
For any scalar $X,Z\in L_2$,
\begin{align*} 
    \left\| \operatorname{Var}(X\mid Y) - \operatorname{Var}(Z\mid Y) \right\|_{L_1} &=\E\left| \E[ (QX)^2 - (QZ)^2 \mid Y ] \right| \\ 
    &\leq \mathbb E\left| (QX)^2-(QZ)^2 \right|\\ 
    &\leq \|Q(X-Z)\|_{L_2} \left\{ \|QX\|_{L_2}+\|QZ\|_{L_2} \right\}\\ 
    &\leq \|X-Z\|_{L_2} \left\{ \|X\|_{L_2}+\|Z\|_{L_2} \right\}. 
\end{align*}
Taking $X=R_\alpha$ and $Z=R_0$ proves convergence in $L_1$.
\end{proof}

\begin{proof}[Proof of \Cref{thm:alpha to 0}]
Let $\overline\omega:=\frac1K\sum_{k=1}^K\omega^k$.
Using weak differentiability of $g$,
\[
g(Y+\sqrt\alpha\,\omega^k)-g(Y)
=
\sqrt\alpha\int_0^1
\nabla g(Y+t\sqrt\alpha\,\omega^k)\omega^k\,\rd t.
\]
We can write
\begin{align*}
\cv_\alpha
&=
\frac1K\sum_{k=1}^K
\left\{
\|Y-g(Y+\sqrt\alpha\,\omega^k)\|_2^2
+
\frac{2}{\sqrt\alpha}
(\omega^k)\tran
\bigl(g(Y+\sqrt\alpha\,\omega^k)-Y\bigr)
\right\}=
R_\alpha+T_\alpha,
\end{align*}
where $R_\alpha$ is the random variable defined in
\Cref{equ:def:R_alpha} and
\[
T_\alpha
:=
\frac{2}{\sqrt\alpha}
\{g(Y)-Y\}\tran\overline\omega.
\]

By \Cref{lem:smooth-small-noise-limit},
$R_\alpha\to R_0$ in $L_2$. Hence there exist
$\alpha_1>0$ and a finite constant $C_R$ such that
\[
\sup_{0<\alpha\leq\alpha_1}
\mathbb E[R_\alpha^2]
\leq C_R.
\]
In particular,
\[
\sup_{0<\alpha\leq\alpha_1}
\mathbb E\!\left[\operatorname{Var}(R_\alpha\mid Y)\right]
\leq C_R.
\]

Under Assumption~\ref{assump:normal-randomization},
\begin{align*}
\operatorname{Var}(\overline\omega)
&=
\frac{1}{K^2}
\left\{
K\sigma^2I_n+
K(K-1)\rho\sigma^2I_n
\right\}=
\frac{\sigma^2}{K}
\{1+(K-1)\rho\}I_n.
\end{align*}

Suppose first that $\rho=-1/(K-1)$. Then
$\operatorname{Var}(\overline\omega)=0$, and since
$\mathbb E[\overline\omega]=0$, it follows that
$\overline\omega=0$ almost surely. Thus $T_\alpha=0$ almost surely, and
\[
\mathbb E\!\left[
\operatorname{Var}(\cv_\alpha\mid Y)
\right]
=
\mathbb E\!\left[
\operatorname{Var}(R_\alpha\mid Y)
\right]
\leq C_R.
\]
This proves the first case.

Suppose that
$\rho>-1/(K-1)$, we then have
$1+(K-1)\rho>0$. Independence of the randomization from $Y$ gives
\begin{align*}
\mathbb E\!\left[
\operatorname{Var}(T_\alpha\mid Y)
\right]
&=
\frac{4}{\alpha}
\mathbb E\!\left[
\{g(Y)-Y\}\tran
\operatorname{Var}(\overline\omega)
\{g(Y)-Y\}
\right]\\
&=
\frac{4\sigma^2}{\alpha K}
\{1+(K-1)\rho\}
\mathbb E\|g(Y)-Y\|_2^2\\
&=:\frac{c}{\alpha}.
\end{align*}
By the assumption that $\E[\|g(Y)- Y\|_2^2 ]>0$, $c>0$.
Applying Cauchy--Schwarz inequality yields
\begin{align*}
\left|
\mathbb E\!\left[
\operatorname{Cov}(R_\alpha,T_\alpha\mid Y)
\right]
\right|
&\leq
\left\{
\mathbb E\!\left[
\operatorname{Var}(R_\alpha\mid Y)
\right]
\mathbb E\!\left[
\operatorname{Var}(T_\alpha\mid Y)
\right]
\right\}^{1/2}\leq
\sqrt{C_Rc}\,\alpha^{-1/2}.
\end{align*}
Therefore,
\begin{align*}
\mathbb E\!\left[
\operatorname{Var}(\cv_\alpha\mid Y)
\right]
&=
\mathbb E\!\left[
\operatorname{Var}(R_\alpha\mid Y)
\right]
+
\frac{c}{\alpha}
+
2\mathbb E\!\left[
\operatorname{Cov}(R_\alpha,T_\alpha\mid Y)
\right]=\Theta(\frac{1}{\alpha}),\qquad \text{as }\alpha\downarrow0.
\end{align*}

\end{proof}

\subsection{Proof of Proposition~\ref{lem: recipe:AR}}
\label{prf:lem:recipe:AR}

\begin{proof}[Proof of \Cref{lem: recipe:AR}]

Exchangeability of $(M^k)$ implies the exchangeability of $(\omega^k)$.

Conditional on $M^k$, $\omega^k \sim \calN(0, \sigma^2 M^k(M^k)^\top) = \calN(0, \sigma^2 I_n)$ by co-isometry~\ref{cond:co-isometry}, so $\omega^k \sim \calN(0, \sigma^2 I_n)$ marginally. By the sum-zero constraint~\ref{cond:sum-zero-matrices}, $\sum_k \omega^k = (\sum_k M^k) Z = 0$. By exchangeability,
$$
0=\Var\Big(\sum_k\omega^k\Big)=K\sigma^2I_n+K(K-1)\operatorname{Cov}(\omega^1,\omega^2),
$$
which implies $\rho = - \frac{1}{K-1}$.
\end{proof}

\subsection{Proof of Theorem~\ref{thm:minimax}}
\label{prf:thm:minimax}

\begin{lemma}[Variance decomposition]
\label{lem:reducible:variance}
Consider the randomization variables $\omega^k=M^kZ$ satisfying
\Cref{lem: recipe:AR} and independent of $Y$. Define
\[
H_g(Y)
:=
\frac12\left\{
\nabla g(Y)+\nabla g(Y)\tran
\right\},\qquad
D
:=
M^1(M^2)\tran+\frac{1}{K-1}I_n.
\]
Then
\begin{align*}
\lim_{\alpha\downarrow0}
\mathbb E\left[
\operatorname{Var}(\cv_\alpha\mid Y)
\right]
&=
\frac{8\sigma^4}{K-1}
\mathbb E\|H_g(Y)\|_{\mathrm F}^2
+
\frac{8(K-1)\sigma^4}{K}
\mathbb E\left[
\operatorname{tr}
\left\{
H_g(Y)D H_g(Y)D\tran
\right\}
\right].
\end{align*}
\end{lemma}

\begin{proof}[Proof of \Cref{lem:reducible:variance}]
By the sum-zero condition $\frac1K\sum_{k=1}^K\omega^k=0$,
the term $T_\alpha$ in the proof of
\Cref{thm:alpha to 0} vanishes and $\cv_\alpha=R_\alpha$, where $R_\alpha$ 
is defined in \Cref{equ:def:R_alpha}. By
\Cref{lem:smooth-small-noise-limit},
\[
\operatorname{Var}(\cv_\alpha\mid Y)
\longrightarrow
\operatorname{Var}\Bigl(
\frac2K\sum_{k=1}^K
(\omega^k)\tran\nabla g(Y)\omega^k
\, \mid \,Y
\Bigr)
\]
in $L_1$. Consequently,
\begin{align}
\lim_{\alpha\downarrow0}
\mathbb E\!\left[
\operatorname{Var}(\cv_\alpha\mid Y)
\right]
=
\mathbb E\Bigl[
\operatorname{Var}\Bigl(
\frac2K\sum_{k=1}^K
(\omega^k)\tran H_g(Y)\omega^k
\,\mid \,Y
\Bigr)
\Bigr],
\label{equ:rvar:limit}
\end{align}
where we used $x^\top A x=x^\top(A+A^\top)x/2$.

Condition on $Y$ and write $H=H_g(Y)$. Thus $H$ is a fixed symmetric
matrix in the following calculation. Set
\[
S_k:=(M^k)\tran H M^k,
\qquad
S:=\sum_{k=1}^KS_k.
\]
Since $\omega^k=M^kZ$,
\[
\frac2K\sum_{k=1}^K(\omega^k)\tran H\omega^k
=
\frac2K Z\tran S Z.
\]
Conditional on $M$, the matrix $S$ is symmetric and
$Z\sim\calN(0,\sigma^2I_d)$. The law of total variance therefore gives
\begin{align*}
V(H)
&:=
\operatorname{Var}\left(
\frac2K\sum_{k=1}^K(\omega^k)\tran H\omega^k
\right)\\
&=
\frac4{K^2}
\left\{
\mathbb E\!\left[
\operatorname{Var}(Z\tran S Z\mid M)
\right]
+
\operatorname{Var}\left(
\mathbb E[Z\tran S Z\mid M]
\right)
\right\}.
\end{align*}
For a symmetric matrix $S$,
\[
\mathbb E[Z\tran S Z\mid M]
=
\sigma^2\operatorname{tr}(S),
\qquad
\operatorname{Var}(Z\tran S Z\mid M)
=
2\sigma^4\|S\|_{\mathrm F}^2.
\]
Moreover, the co-isometry condition
$M^k(M^k)\tran=I_n$ gives
\[
\operatorname{tr}(S)
=
\sum_{k=1}^K
\operatorname{tr}\left\{
H M^k(M^k)\tran
\right\}
=
K\operatorname{tr}(H).
\]
Thus $\operatorname{tr}(S)$ is deterministic and
\begin{equation}
V(H)
=
\frac{8\sigma^4}{K^2}
\mathbb E\|S\|_{\mathrm F}^2.
\label{eq:V-corrected}
\end{equation}

Let $P_{jk}:=M^j(M^k)\tran$.
Since $H$ is symmetric,
\begin{align*}
\|S\|_{\mathrm F}^2
&=
\sum_{j,k=1}^K\operatorname{tr}(S_jS_k)=
\sum_{j,k=1}^K
\operatorname{tr}
\left(
H P_{jk}H P_{jk}\tran
\right).
\end{align*}
Using $P_{jj}=I_n$ and exchangeability,
\begin{equation}
\mathbb E\|S\|_{\mathrm F}^2
=
K\|H\|_{\mathrm F}^2
+
K(K-1)
\mathbb E\operatorname{tr}
\left(
H P_{12}H P_{12}\tran
\right).
\label{eq:S2-corrected}
\end{equation}

The sum-zero condition implies
\[
\sum_{k\neq1}P_{1k}
=
M^1\left(\sum_{k\neq1}M^k\right)\tran
=
-M^1(M^1)\tran
=
-I_n.
\]
By exchangeability,
$\mathbb E[P_{12}]=-\frac1{K-1}I_n$.
Writing
\[
P_{12}
=
-\frac1{K-1}I_n+D,
\qquad
\mathbb E[D]=0,
\]
we obtain
\begin{align*}
\mathbb E\operatorname{tr}
\left(
H P_{12}H P_{12}\tran
\right)
&=
\frac1{(K-1)^2}\operatorname{tr}(H^2)
+
\mathbb E\operatorname{tr}
\left(
H D H D\tran
\right),
\end{align*}
because the two terms linear in $D$ have expectation zero.

Substituting into \eqref{eq:S2-corrected} gives
\[
\mathbb E\|S\|_{\mathrm F}^2
=
\frac{K^2}{K-1}\|H\|_{\mathrm F}^2
+
K(K-1)
\mathbb E\operatorname{tr}
\left(
H D H D\tran
\right).
\]
Combining this identity with \eqref{eq:V-corrected},
\begin{align*}
V(H)
&=
\frac{8\sigma^4}{K-1}\|H\|_{\mathrm F}^2+
\frac{8(K-1)\sigma^4}{K}
\mathbb E\operatorname{tr}
\left(
H D H D\tran
\right).
\end{align*}
Finally, substitute $H=H_g(Y)$ and take expectation over $Y$ in
\eqref{equ:rvar:limit}.
\end{proof}

\begin{proof}[Proof of \Cref{thm:minimax}]
For a construction $M$, define
\[
\mathcal V(M;g)
:=
\lim_{\alpha\downarrow0}
\mathbb E\!\left[
\operatorname{Var}(\cv_\alpha\mid Y)
\right].
\]
Let
\[
H_g(Y)
=
\frac12\left\{
\nabla g(Y)+\nabla g(Y)\tran
\right\},
\qquad
D
=
M^1(M^2)\tran+\frac1{K-1}I_n.
\]
By \Cref{lem:reducible:variance},
\begin{align*}
\mathcal V(M;g)
&=
\frac{8\sigma^4}{K-1}
\mathbb E\|H_g(Y)\|_{\mathrm F}^2
+
\frac{8(K-1)\sigma^4}{K}
\mathbb E\operatorname{tr}
\left\{
H_g(Y)D H_g(Y)D\tran
\right\}.
\end{align*}

Consider the test function
$g_\star(y)=\frac1{\sqrt n}y+b$ for some fixed $b\neq 0$,
which satisfies
\[
\nabla g_\star(y)
=
H_{g_\star}(y)
=
\frac1{\sqrt n}I_n,
\qquad
\|\nabla g_\star(y)\|_{\mathrm F}^2=1,\qquad \E[\|g_\star(Y)-Y\|_2^2]> 0.
\]
Therefore,
\begin{align*}
\mathcal V(M;g_\star)
&=
\frac{8\sigma^4}{K-1}
+
\frac{8(K-1)\sigma^4}{Kn}
\mathbb E\|D\|_{\mathrm F}^2.
\end{align*}
It follows that, for every admissible construction $M$,
\[
\sup_{g\in\mathcal{G}}
\mathcal V(M;g)
\geq
\frac{8\sigma^4}{K-1}
+
\frac{8(K-1)\sigma^4}{Kn}
\mathbb E\|D\|_{\mathrm F}^2
\geq
\frac{8\sigma^4}{K-1}.
\]
Consequently,
\[
\inf_{M\in\mathcal{M}}\;
\sup_{g\in\mathcal{G}}\;
\mathcal V(M;g)
\geq
\frac{8\sigma^4}{K-1}.
\]

For the jointly normal antithetic construction,
\[
M^j(M^k)\tran
=
-\frac1{K-1}I_n
\qquad
\text{for every }j\neq k,
\]
so $D=0$ almost surely. In that case,
\[
\mathcal V(M;g)
=
\frac{8\sigma^4}{K-1}
\mathbb E\|H_g(Y)\|_{\mathrm F}^2.
\]
Since
\[
\|H_g(Y)\|_{\mathrm F}
\leq
\|\nabla g(Y)\|_{\mathrm F},
\]
we have, under the constraint
$\mathbb E\|\nabla g(Y)\|_{\mathrm F}^2=1$,
\[
\mathcal V(M;g)
\leq
\frac{8\sigma^4}{K-1}.
\]
The function $g_\star$ attains equality. Hence
\[
\inf_{M\in\mathcal{M}}\;
\sup_{g\in\mathcal{G}}\;
\mathcal V(M;g)
=
\frac{8\sigma^4}{K-1}.
\]

It remains to characterize equality. If an admissible construction
satisfies
\[
\sup_{g\in\mathcal{G}}\;
\mathcal V(M;g)
=
\frac{8\sigma^4}{K-1},
\]
then the lower bound obtained using $g_\star$ forces
$\mathbb E\|D\|_{\mathrm F}^2=0$.
Thus
$M^1(M^2)\tran=-\frac1{K-1}I_n$ almost surely.
By exchangeability, the same identity holds almost surely for every
$j\neq k$. Conversely, if these identities hold, then $D=0$ and the
preceding argument shows that the minimax value is attained.

Finally, these identities are equivalent to
joint normality of $W=(\omega^1,\ldots,\omega^K)$. To see this, let
$\overline M$ denote the vertical stacking of $M^1,\ldots,M^K$ and set
\[
G_M:=\overline M\,\overline M\tran.
\]
Conditional on $M$,
\[
(\omega^1,\ldots,\omega^K)\mid M
\sim
\calN(0,\sigma^2G_M).
\]
If
$M^j(M^k)\tran=-I_n/(K-1)$ almost surely for all $j\neq k$, then
$G_M$ is deterministic, and hence the unconditional stacked vector is
jointly normal.

Conversely, suppose that the stacked vector is jointly normal.
Its covariance matrix is then
\begin{align*}
    \Cov[W] = \Cov[\E[W\mid M]] + \E[\Cov[W\mid M]] = 0 + \sigma^2 \E[G_M] = \sigma^2 \E[G_M],
\end{align*}
and therefore its characteristic function is $\exp\{-\frac{\sigma^2}{2}t\tran\mathbb E[G_M]t \}$.
On the other hand, conditional on $M$, $W\mid M$ is normal with covariance $\sigma^2 G_M$, so its characteristic function can be written as $\mathbb E\exp\{-\frac{\sigma^2}{2}t\tran G_Mt\}$.
Consequently, for every $t\in\mathbb R^{Kn}$,
\[
\mathbb E\exp\left\{
-\frac{\sigma^2}{2}t\tran G_Mt
\right\}
=
\exp\left\{
-\frac{\sigma^2}{2}t\tran\mathbb E[G_M]t
\right\}.
\]
Since $x\mapsto e^{-x}$ is strictly convex, equality in Jensen's
inequality implies that $t\tran G_Mt$ is almost surely constant.
The off-diagonal blocks of $G_M$ are therefore
$-I_n/(K-1)$ almost surely. This proves the claimed equivalence and
completes the proof.
\end{proof}

\subsection{Proof of Theorem~\ref{thm:reducible-variance-nonsmooth}}
\label{prf:thm:reducible-variance-nonsmooth}

\begin{proof}[Proof of \Cref{thm:reducible-variance-nonsmooth}]
Set $\varepsilon=\sqrt{\alpha}$,
$\overline\omega=\frac1K\sum_{k=1}^K\omega^k$.
We have
\begin{equation}
\label{eq:nonsmooth-general-decomposition}
\cv_\alpha
=
R_\varepsilon+T_\varepsilon,
\end{equation}
where
\begin{align*}
R_\varepsilon
&=
A_\varepsilon+B_\varepsilon,\qquad
A_\varepsilon=
\frac1K\sum_{k=1}^K
\|Y-g(Y+\varepsilon\omega^k)\|_2^2,\\
B_\varepsilon
&=
\frac{2}{K\varepsilon}
\sum_{k=1}^K
\left\langle
\omega^k,
g(Y+\varepsilon\omega^k)-g(Y)
\right\rangle,\qquad
T_\varepsilon
=
\frac{2}{\varepsilon}
\left\langle
\overline\omega,
g(Y)-Y
\right\rangle.
\end{align*}

As established in the proof of \Cref{thm:alpha to 0},
the moment assumptions and the Gaussian-shift bound imply
\[
\mathbb E[A_\varepsilon^2]=O(1).
\]
Moreover, by \Cref{lem:small-noise-increment},
\begin{align*}
\mathbb E[B_\varepsilon^2]
&\leq
\frac{4}{K\varepsilon^2}
\sum_{k=1}^K
\mathbb E\!\left[
\|\omega^k\|_2^2
\|g(Y+\varepsilon\omega^k)-g(Y)\|_2^2
\right]=
O(\varepsilon^{-1}).
\end{align*}
Consequently,
\begin{equation}
\label{eq:nonsmooth-regular-rvar}
\mathbb E\!\left[
\operatorname{Var}(R_\varepsilon\mid Y)
\right]
\leq
\mathbb E[R_\varepsilon^2]
=
O(\varepsilon^{-1}).
\end{equation}

Suppose first that
$\rho=-\frac{1}{K-1}$.
Then \(\overline\omega=0\) almost surely, so
\(T_\varepsilon=0\) almost surely. Hence,
by \eqref{eq:nonsmooth-general-decomposition} and
\eqref{eq:nonsmooth-regular-rvar},
\[
\mathbb E\!\left[
\operatorname{Var}(\cv_\alpha\mid Y)
\right]
=
O(\varepsilon^{-1})
=
O(\alpha^{-1/2}).
\]

Now suppose that
$\rho>-\frac{1}{K-1}$.
Assumption~\ref{assump:normal-randomization} gives
\[
\operatorname{Var}(\overline\omega)
=
\frac{\sigma^2}{K}
\{1+(K-1)\rho\}I_n.
\]
The assumption that $\E[\|g(Y)-Y\|^2]>0$ implies
\begin{equation}
\label{eq:nonsmooth-singular-rvar}
\mathbb E\!\left[
\operatorname{Var}(T_\varepsilon\mid Y)
\right]
=
\frac{c_\rho}{\varepsilon^2}
\end{equation}
for some constant \(c_\rho>0\).

The upper bound follows from
\[
\operatorname{Var}(X+Z\mid Y)
\leq
2\operatorname{Var}(X\mid Y)
+
2\operatorname{Var}(Z\mid Y),
\]
which, together with
\eqref{eq:nonsmooth-regular-rvar} and
\eqref{eq:nonsmooth-singular-rvar}, yields
\[
\mathbb E\!\left[
\operatorname{Var}(\cv_\alpha\mid Y)
\right]
=
O(\varepsilon^{-2}).
\]

For the lower bound, conditional Cauchy--Schwarz gives
\begin{align*}
\left|
\mathbb E\!\left[
\operatorname{Cov}
\left(
R_\varepsilon,T_\varepsilon\mid Y
\right)
\right]
\right|
&\leq
\left\{
\mathbb E[
\operatorname{Var}(R_\varepsilon\mid Y)]
\,
\mathbb E[
\operatorname{Var}(T_\varepsilon\mid Y)]
\right\}^{1/2}=
O(\varepsilon^{-3/2}).
\end{align*}
Therefore,
\begin{align*}
\mathbb E\!\left[
\operatorname{Var}(\cv_\alpha\mid Y)
\right]
&\geq
\mathbb E\!\left[
\operatorname{Var}(T_\varepsilon\mid Y)
\right]
-
2\left|
\mathbb E\!\left[
\operatorname{Cov}
\left(
R_\varepsilon,T_\varepsilon\mid Y
\right)
\right]
\right|\\
&\geq
\frac{c_\rho}{\varepsilon^2}
-
O(\varepsilon^{-3/2})
=
\Omega(\varepsilon^{-2}).
\end{align*}
Combining the upper and lower bounds gives
\[
\mathbb E\!\left[
\operatorname{Var}(\cv_\alpha\mid Y)
\right]
=
\Theta(\varepsilon^{-2})
=
\Theta(\alpha^{-1}),
\]
which completes the proof.
\end{proof}

\begin{lemma}[Small-noise increment bound]
\label{lem:small-noise-increment}
Let $\omega\sim \calN(0,\sigma^2 I_n)$ be independent of $Y$.
Under Assumption~\ref{ass:beyond smooth},
\[
\mathbb E\left[
\|\omega\|^2
\|g(Y+\ep\omega)-g(Y)\|^2
\right]
=
O(\ep),
\qquad \ep\downarrow0 .
\]
\end{lemma}

\begin{proof}[Proof of Lemma~\ref{lem:small-noise-increment}]
Write
\[
I_\ell(y)=\mathbf 1\{h_\ell(y)>0\}.
\]
Then
\begin{align*}
g(Y+\ep\omega)-g(Y)
={}&
f_0(Y+\ep\omega)-f_0(Y)\\
&+
\sum_{\ell=1}^m
\{a_\ell(Y+\ep\omega)-a_\ell(Y)\}
I_\ell(Y+\ep\omega)\\
&+
\sum_{\ell=1}^m
a_\ell(Y)
\{I_\ell(Y+\ep\omega)-I_\ell(Y)\}.
\end{align*}

Since $f_0$ is weakly differentiable, we have
\[
f_0(Y+\ep\omega)-f_0(Y)
=
\ep
\int_0^1
\nabla f_0(Y+t\ep\omega)\omega\,\rd t.
\]
By Jensen's inequality and the Cauchy--Schwarz inequality, we have
\begin{align*}
    \|f_0(Y+\ep\omega)-f_0(Y)\|^2 &\leq \ep^2 \int_0^1 \|\nabla f_0(Y+t\ep\omega)\omega\|^2\,\rd t \\
    &\leq \ep^2 \int_0^1 \|\nabla f_0(Y+t\ep\omega)\|_{\rm F}^2\, \|\omega\|^2\,\rd t.
\end{align*}
Hence,
\begin{align*}
\mathbb E\left[
\|\omega\|^2
\|f_0(Y+\ep\omega)-f_0(Y)\|^2
\right] 
&\leq
\ep^2
\int_0^1
\mathbb E\left[
\|\omega\|^4
\|\nabla f_0(Y+t\ep\omega)\|_{\rm F}^2
\right]\rd t
=
O(\ep^2),
\end{align*}
where the last step follows from \Cref{lem:gaussian-shift-continuity}, \Cref{eq:uniform-gaussian-shift-bound} applied with $p=2,m=4,f=\nabla f_0$. The same argument gives
\[
\mathbb E\left[
\|\omega\|^2
\|a_\ell(Y+\ep\omega)-a_\ell(Y)\|^2
\right]
=
O(\ep^2),
\]
and thus
\begin{align*}
    \EE{\|\omega\|^2 \|(a_\ell(Y+\ep\omega)-a_\ell(Y)) I_\ell(Y+\ep\omega)\|^2} = O(\ep^2).
\end{align*}

Let \(L_\ell\) be the Lipschitz constant of \(h_\ell\), and define
\[
\mathcal C_{\ell,\ep}
=
\{I_\ell(Y+\ep\omega)\neq I_\ell(Y)\}.
\]
If \(\mathcal C_{\ell,\ep}\) occurs, then $h_\ell(Y)$ and $h_\ell(Y+\ep\omega)$ have opposite signs, and hence
\[
|h_\ell(Y)|
\leq
|h_\ell(Y+\ep\omega)-h_\ell(Y)|
\leq
L_\ell\ep\|\omega\|.
\]
Thus
\[
\mathcal C_{\ell,\ep}
\subseteq
\{|h_\ell(Y)|\leq L_\ell\ep\|\omega\|\}.
\]
We then have
\begin{align*}
\mathbb E\left[
\|\omega\|^2
\|a_\ell(Y)\|^2
\mathbf 1_{\mathcal C_{\ell,\ep}}
\right] &\leq \EE{\|\omega\|^2 \|a_\ell(Y)\|^2 \Indc{|h_\ell(Y)|\leq L_\ell \ep \|\omega\| } }\\
&=  \EE{\|\omega\|^2 \|a_\ell(Y)\|^2 \Indc{|h_\ell(Y)|\leq L_\ell \ep \|\omega\|, L_\ell \ep \|\omega\|\leq t_0 } }\\
&\qquad + \EE{\|\omega\|^2 \|a_\ell(Y)\|^2 \Indc{|h_\ell(Y)|\leq L_\ell \ep \|\omega\|, L_\ell \ep \|\omega\|> t_0 } }.
\end{align*}
On the event \(\{L_\ell\ep\|\omega\|\leq t_0\}\), 
by condition~\eqref{eq:bounded-jump}, 
\begin{align*}
\mathbb E\left[
\|a_\ell(Y)\|^2
\mathbf 1\{|h_\ell(Y)|\leq L_\ell \ep \|\omega\| \}
\mid \omega
\right]
&=
\int_{-L_\ell \ep \|\omega\|}^{L_\ell \ep \|\omega\|}
\mathbb E\{\|a_\ell(Y)\|^2\mid h_\ell(Y)=s\}
p_\ell(s)\,\rd s  \leq C L_\ell \ep \|\omega\|.
\end{align*}
Hence,
\begin{align*}
    \EE{\|\omega\|^2 \|a_\ell(Y)\|^2 \Indc{|h_\ell(Y)|\leq L_\ell \ep \|\omega\|, L_\ell \ep \|\omega\|\leq t_0 } } &= O(\ep).
\end{align*}
On the event \( \{L_\ell\ep\|\omega\|> t_0\} \),
\begin{align*}
&\EE{\|\omega\|^2 \|a_\ell(Y)\|^2 \Indc{|h_\ell(Y)|\leq L_\ell \ep \|\omega\|, L_\ell \ep \|\omega\|> t_0 } }\\
&\qquad \leq\mathbb E\|a_\ell(Y)\|^2
\mathbb E\left[
\|\omega\|^2
\mathbf 1\{L_\ell\ep\|\omega\|>t_0\}
\right] \\
&\qquad \leq \mathbb E\|a_\ell(Y)\|^2 \cdot \EE{\|\omega\|^3 \frac{L_\ell \ep}{t_0} } \\
&\qquad=O(\ep).
\end{align*}
This proves 
\begin{align*}
    \EE{\|\omega\|^2 \|a_\ell(Y)\|^2 \mathbf 1_{\mathcal C_{\ell,\ep}}} = O(\ep).
\end{align*}
\end{proof}

\subsection{Proof of Theorem~\ref{thm:controlvariate}}
\label{prf:thm:controlvariate}
\begin{proof}[Proof of \Cref{thm:controlvariate}]
Set $\varepsilon=\sqrt{\alpha}$,
$U_k=Y+\varepsilon\omega^k$,
and $\overline\omega=\frac1K\sum_{k=1}^K\omega^k$.
For \(k=1,\ldots,K\) and \(\ell=1,\ldots,m\), define
\[
I_{\ell k,\varepsilon}
=
\Indc{h_\ell(U_k)>0},\qquad
q_{\ell k,\varepsilon}(Y)
=
\mathbb E\!\left[
\omega^k I_{\ell k,\varepsilon}
\,\middle|\,
Y
\right].
\]
Thus,
\[
r_\ell(\omega^k,Y)
=
q_{\ell k,\varepsilon}(Y)
-
\omega^k I_{\ell k,\varepsilon}.
\]
We can write
\begin{align*}
\cv_\alpha
={}&
A_\varepsilon
+
\frac{2}{K\varepsilon}
\sum_{k=1}^K
\left\langle
\omega^k,f_0(U_k)-Y
\right\rangle +
\frac{2}{K\varepsilon}
\sum_{k=1}^K\sum_{\ell=1}^m
\left\langle
\omega^k,a_\ell(U_k)
\right\rangle
I_{\ell k,\varepsilon},
\end{align*}
where
\[
A_\varepsilon
=
\frac1K\sum_{k=1}^K
\|Y-g(U_k)\|_2^2.
\]

Adding the control variate yields
\begin{align*}
\widetilde{\cv}_{\alpha}
={}&
A_\varepsilon
+
\frac{2}{K\varepsilon}
\sum_{k=1}^K
\left\langle
\omega^k,f_0(U_k)-Y
\right\rangle 
+
\frac{2}{K\varepsilon}
\sum_{k=1}^K\sum_{\ell=1}^m
\left\langle
\omega^k,
a_\ell(U_k)-a_\ell(Y)
\right\rangle
I_{\ell k,\varepsilon}
+
D_\varepsilon(Y),
\end{align*}
where
\[
D_\varepsilon(Y)
=
\frac{2}{K\varepsilon}
\sum_{k=1}^K\sum_{\ell=1}^m
\left\langle
a_\ell(Y),
q_{\ell k,\varepsilon}(Y)
\right\rangle.
\]
In particular, \(D_\varepsilon(Y)\) is measurable with respect to
\(Y\).

By weak differentiability,
\begin{align*}
f_0(U_k)-f_0(Y)
&=
\varepsilon
\int_0^1
\nabla f_0(Y+t\varepsilon\omega^k)\omega^k
\,\rd t
\\
a_\ell(U_k)-a_\ell(Y)
&=
\varepsilon
\int_0^1
\nabla a_\ell(Y+t\varepsilon\omega^k)\omega^k
\,\rd t.
\end{align*}
Therefore,
\begin{equation}
\label{eq:control-variate-decomposition}
\widetilde{\cv}_{\alpha}
=
D_\varepsilon(Y)+R_\varepsilon+T_\varepsilon,
\end{equation}
where
\begin{align*}
R_\varepsilon
={}&
A_\varepsilon
+
\frac2K\sum_{k=1}^K
\int_0^1
(\omega^k)\tran
\nabla f_0(Y+t\varepsilon\omega^k)
\omega^k
\,\rd t
\notag\\
&+
\frac2K
\sum_{k=1}^K\sum_{\ell=1}^m
I_{\ell k,\varepsilon}
\int_0^1
(\omega^k)\tran
\nabla a_\ell(Y+t\varepsilon\omega^k)
\omega^k
\,\rd t
\end{align*}
and
\begin{equation*}
T_\varepsilon
=
\frac{2}{\varepsilon}
\left\langle
\overline\omega,
f_0(Y)-Y
\right\rangle.
\end{equation*}
Since \(D_\varepsilon(Y)\) is \(Y\)-measurable,
\begin{equation*}
\operatorname{Var}
\left(
\widetilde{\cv}_{\alpha}\mid Y
\right)
=
\operatorname{Var}
\left(
R_\varepsilon+T_\varepsilon\mid Y
\right).
\end{equation*}

Following the same Jensen and Gaussian-shift argument used to control
the regular term in the proof of \Cref{thm:alpha to 0}, applied
componentwise to \(f_0,a_1,\ldots,a_m\), and using
\(I_{\ell k,\varepsilon}\leq 1\), there exist
\(\varepsilon_0>0\) and \(C_R<\infty\) such that
\begin{equation}
\label{eq:control-variate-regular-bound}
\sup_{0<\varepsilon\leq\varepsilon_0}
\mathbb E[R_\varepsilon^2]
\leq C_R.
\end{equation}
Consequently,
\[
\sup_{0<\varepsilon\leq\varepsilon_0}
\mathbb E\!\left[
\operatorname{Var}(R_\varepsilon\mid Y)
\right]
\leq C_R.
\]

The remainder of the argument is identical to the necessity and
sufficiency argument in the proof of \Cref{thm:alpha to 0}, with
\(g(Y)-Y\) replaced by \(f_0(Y)-Y\). Indeed,
$\operatorname{Var}(\overline\omega)= \frac{\sigma^2}{K}
\{1+(K-1)\rho\}I_n$.
If $\rho=-\frac{1}{K-1}$,
then \(\overline\omega=0\) almost surely, and hence
\(T_\varepsilon=0\) almost surely. Since \(D_\varepsilon(Y)\) is
\(Y\)-measurable, \eqref{eq:control-variate-decomposition} and
\eqref{eq:control-variate-regular-bound} give
\[
\mathbb E\!\left[
\operatorname{Var}
\left(
\widetilde{\cv}_{\alpha}\mid Y
\right)
\right]
=
\mathbb E\!\left[
\operatorname{Var}(R_\varepsilon\mid Y)
\right]
\leq C_R.
\]
Conversely, if $\rho\neq-\frac{1}{K-1}$,
then the nondegeneracy condition $\E[\|f_0(Y) - Y\|_2^2 ]>0$ implies that
\[
\mathbb E\!\left[
\operatorname{Var}(T_\varepsilon\mid Y)
\right]
=
\frac{c_\rho}{\varepsilon^2}
\]
for some \(c_\rho>0\). By conditional Cauchy--Schwarz and
\eqref{eq:control-variate-regular-bound},
\[
\left|
\mathbb E\!\left[
\operatorname{Cov}
\left(
R_\varepsilon,T_\varepsilon\mid Y
\right)
\right]
\right|
=
O(\varepsilon^{-1}).
\]
Therefore,
\[
\mathbb E\!\left[
\operatorname{Var}
\left(
\widetilde{\cv}_{\alpha}\mid Y
\right)
\right]
\geq
\frac{c_\rho}{\varepsilon^2}
-
O(\varepsilon^{-1})
\longrightarrow\infty.
\]
Thus the reducible variance is \(O(1)\) as
\(\alpha=\varepsilon^2\downarrow0\) if and only if
\(\rho=-1/(K-1)\).
\end{proof}

\subsection{Technical lemma}

\begin{lemma}[Gaussian-shift continuity]
\label{lem:gaussian-shift-continuity}
Let $Y\sim\calN(\mu,\sigma^2I_n)$ and
$W\sim\calN(0,\sigma^2I_n)$ be independent, and set
$\gamma_\eta=\calN(\mu,(1+\eta)\sigma^2I_n)$ for $\eta>0$.
If $f:\mathbb R^n\to\mathbb R^d$ is measurable and
$f\in L_p(\gamma_\eta)$ for some $p\geq1$, then, for every fixed
$m\geq0$,
\[
\lim_{\varepsilon\downarrow0}\;
\sup_{0\leq r\leq\varepsilon}
\mathbb E\left[
\|W\|_2^m\|f(Y+rW)-f(Y)\|_2^p
\right]=0.
\]
The same conclusion holds for matrix-valued $f$, with the Frobenius norm.
\end{lemma}

\begin{proof}
For $r\geq0$, let $U_r=Y+rW$. Then
$U_r\sim\gamma_{r^2}$ and
\[
W\mid U_r=u\sim\calN\!\left(
\frac{r}{1+r^2}(u-\mu),
\frac{\sigma^2}{1+r^2}I_n
\right).
\]
Consequently, for $0\leq r\leq\sqrt{\eta/2}$,
\[
\mathbb E(\|W\|_2^m\mid U_r=u)
\leq C_m(1+\|u-\mu\|_2^m).
\]
The strict variance gap $1+r^2\leq1+\eta/2<1+\eta$ also gives
\[
\sup_{\substack{0\leq r\leq\sqrt{\eta/2}\\u\in\mathbb R^n}}
(1+\|u-\mu\|_2^m)
\frac{\mathrm d\gamma_{r^2}}{\mathrm d\gamma_\eta}(u)
\leq C_{m,\eta}.
\]
It follows that, for every $h\in L_p(\gamma_\eta)$,
\begin{equation}
\label{eq:uniform-gaussian-shift-bound}
\sup_{0\leq r\leq\sqrt{\eta/2}}
\mathbb E\!\left[
\|W\|_2^m\|h(Y+rW)\|_2^p
\right]
\leq C\|h\|_{L_p(\gamma_\eta)}^p.
\end{equation}
The same bound at $r=0$ controls
$\mathbb E[\|W\|_2^m\|h(Y)\|_2^p]$.

Choose $f_j\in C_c(\mathbb R^n;\mathbb R^d)$ such that
$f_j\to f$ in $L_p(\gamma_\eta)$. For fixed $j$, uniform continuity gives
\[
\sup_{0\leq r\leq\varepsilon}
\|f_j(Y+rW)-f_j(Y)\|_2^p\longrightarrow0
\quad\text{almost surely}.
\]
After multiplication by $\|W\|_2^m$, this is bounded by
$2^p\|f_j\|_\infty^p\|W\|_2^m$, so dominated convergence applies.
Finally, write $h_j=f-f_j$ and use
\[
\|f(Y+rW)-f(Y)\|_2^p
\leq 3^{p-1}\!\left
\{
\|h_j(Y+rW)\|_2^p
+\|f_j(Y+rW)-f_j(Y)\|_2^p
+\|h_j(Y)\|_2^p
\right\}.
\]
Taking the supremum over $0\leq r\leq\varepsilon$, applying
\eqref{eq:uniform-gaussian-shift-bound}, and then letting first
$\varepsilon\downarrow0$ and subsequently $j\to\infty$ proves the claim.
\end{proof}

\section{Control variate for hard-thresholded ridge regression}
\label{app:hard-thresholded-ridge-control-variate}

We derive the control variate associated with
hard-thresholded ridge estimator. Define
\begin{equation*}
C_{\lambda}
=
(X\tran X+\lambda I_p)^{-1}X\tran
\in\mathbb R^{p\times n}.
\end{equation*}
Let \(c_j\tran\) be the \(j\)th row of \(C_{\lambda}\), let
\(x_j=Xe_j\) be the \(j\)th column of \(X\), and set
\[
\mathcal J_{\lambda}
=
\{j\in\{1,\ldots,p\}:\|c_j\|_2>0\}.
\]
If \(j\notin\mathcal J_{\lambda}\), then \(c_j\tran y=0\) for every
\(y\), so the \(j\)th coordinate contributes nothing and may be
omitted.

\medskip\noindent
\emph{Representation in the non-smooth class.}
For \(j\in\mathcal J_{\lambda}\), define
\begin{equation*}
a_j(y)=x_j(c_j\tran y),
\qquad
h_{j,+}(y)=c_j\tran y-\tau_j,
\qquad
h_{j,-}(y)=-c_j\tran y-\tau_j.
\end{equation*}
Then
\begin{align}
\label{eq:hard-thresholded-ridge-piecewise}
g^{\mathrm{HT}}_{\lambda,\tau}(y)
&=
\sum_{j\in\mathcal J_{\lambda}}
a_j(y)
\left\{
\Indc{h_{j,+}(y)>0}
+
\Indc{h_{j,-}(y)>0}
\right\}
\notag\\
&=
\sum_{j\in\mathcal J_{\lambda}}
x_j(c_j\tran y)\Indc{|c_j\tran y|>\tau_j}.
\end{align}
Thus \(f_0=0\), and each coordinate contributes two threshold terms.

The maps \(a_j\) are linear, while the threshold functions
\(h_{j,+}\) and \(h_{j,-}\) are affine. Moreover,
\[
h_{j,+}(Y)
\sim
\calN(c_j\tran\mu-\tau_j,\sigma^2\|c_j\|_2^2),
\qquad
h_{j,-}(Y)
\sim
\calN(-c_j\tran\mu-\tau_j,\sigma^2\|c_j\|_2^2),
\]
so both threshold scores have bounded densities. Conditional on either
\(h_{j,+}(Y)=s\) or \(h_{j,-}(Y)=s\),
\begin{equation*}
\|a_j(Y)\|_2^2
=
\|x_j\|_2^2(s+\tau_j)^2.
\end{equation*}
Consequently, the conditional second moment in
\eqref{eq:bounded-jump} is bounded on every compact neighborhood of
zero, and the hard-thresholded ridge estimator satisfies
Assumption~\ref{ass:beyond smooth}.

\medskip\noindent
\emph{A Gaussian half-space identity.}
Let
\[
\omega\sim\calN(0,\sigma^2I_n),
\qquad
c\in\mathbb R^n\setminus\{0\},
\qquad
b\in\mathbb R,
\]
and set \(T=c\tran\omega\). Then
\[
T\sim\calN(0,\sigma^2\|c\|_2^2),
\qquad
\mathbb E(\omega\mid T)
=
\frac{c}{\|c\|_2^2}T.
\]
Let
\[
\varphi(u)
=
(2\pi)^{-1/2}\exp(-u^2/2)
\]
denote the standard normal density. For \(Z\sim\calN(0,1)\),
\[
\mathbb E\{Z\Indc{Z>u}\}
=
\varphi(u),
\qquad
\mathbb E\{Z\Indc{Z<u}\}
=
-\varphi(u).
\]
It follows that
\begin{align}
\label{eq:gaussian-upper-halfspace-moment}
\mathbb E\!\left[
\omega\Indc{c\tran\omega>b}
\right]
&=
\frac{\sigma c}{\|c\|_2}
\varphi\!\left(
\frac{b}{\sigma\|c\|_2}
\right),
\\
\label{eq:gaussian-lower-halfspace-moment}
\mathbb E\!\left[
\omega\Indc{c\tran\omega<b}
\right]
&=
-\frac{\sigma c}{\|c\|_2}
\varphi\!\left(
\frac{b}{\sigma\|c\|_2}
\right).
\end{align}

\medskip\noindent
\emph{Conditional moment for the threshold indicators.}
Fix \(j\in\mathcal J_{\lambda}\) and condition on \(Y\). Write
$z_j=c_j\tran Y$.
For the \(k\)th randomized training fold, define
\begin{align*}
I_{j,+,\alpha}^k
&=
\Indc{
c_j\tran(Y+\sqrt\alpha\,\omega^k)>\tau_j
}
=
\Indc{
c_j\tran\omega^k>
(\tau_j-z_j)/\sqrt\alpha
},
\\
I_{j,-,\alpha}^k
&=
\Indc{
c_j\tran(Y+\sqrt\alpha\,\omega^k)<-\tau_j
}
=
\Indc{
c_j\tran\omega^k<
(-\tau_j-z_j)/\sqrt\alpha
}.
\end{align*}

By \eqref{eq:gaussian-upper-halfspace-moment},
\begin{equation*}
q_{j,+,\alpha}(Y)
:=
\mathbb E\!\left[
\omega^k I_{j,+,\alpha}^k
\,\middle|\,
Y
\right]
=
\frac{\sigma c_j}{\|c_j\|_2}
\varphi\!\left(
\frac{\tau_j-z_j}
{\sigma\sqrt\alpha\,\|c_j\|_2}
\right).
\end{equation*}
Similarly, by \eqref{eq:gaussian-lower-halfspace-moment} and the
symmetry of \(\varphi\),
\begin{equation*}
q_{j,-,\alpha}(Y)
:=
\mathbb E\!\left[
\omega^k I_{j,-,\alpha}^k
\,\middle|\,
Y
\right]
=
-\frac{\sigma c_j}{\|c_j\|_2}
\varphi\!\left(
\frac{\tau_j+z_j}
{\sigma\sqrt\alpha\,\|c_j\|_2}
\right).
\end{equation*}

Because \(\tau_j>0\), the upper- and lower-threshold events are
disjoint. Define
\begin{equation*}
I_{j,\alpha}^k
=
I_{j,+,\alpha}^k+I_{j,-,\alpha}^k
=
\Indc{
|z_j+\sqrt\alpha\,c_j\tran\omega^k|>\tau_j
}
\end{equation*}
and
\begin{align*}
q_{j,\alpha}(Y)
&=
q_{j,+,\alpha}(Y)+q_{j,-,\alpha}(Y)
=
\frac{\sigma c_j}{\|c_j\|_2}
\left\{
\varphi\!\left(
\frac{\tau_j-z_j}
{\sigma\sqrt\alpha\,\|c_j\|_2}
\right)
-
\varphi\!\left(
\frac{\tau_j+z_j}
{\sigma\sqrt\alpha\,\|c_j\|_2}
\right)
\right\}.
\end{align*}
Thus,
\begin{equation}
\label{eq:hard-threshold-q-identity}
q_{j,\alpha}(Y)
=
\mathbb E\!\left[
\omega^k I_{j,\alpha}^k
\,\middle|\,
Y
\right].
\end{equation}

\medskip\noindent
\emph{Specialized control variate.}
In the representation
\eqref{eq:hard-thresholded-ridge-piecewise}, the jump-amplitude
function for both one-sided indicators is
$a_j(Y)=x_jz_j$.
Therefore, the two one-sided terms in the general control variate
\eqref{eqn: controlvariate} combine to give
\begin{equation}
\label{eq:hard-thresholded-ridge-control-variate}
\mathcal V^{\mathrm{HT}}_{\alpha}
=
\frac{2}{K\sqrt\alpha}
\sum_{k=1}^K
\sum_{j\in\mathcal J_{\lambda}}
\left\langle
x_jz_j,\,
q_{j,\alpha}(Y)-\omega^k I_{j,\alpha}^k
\right\rangle.
\end{equation}
The control-variate-adjusted estimator is
\begin{equation*}
\widetilde{\cv}^{\mathrm{HT}}_{\alpha}
=
\cv^{\mathrm{HT}}_{\alpha}
+
\mathcal V^{\mathrm{HT}}_{\alpha},
\end{equation*}
where \(\cv^{\mathrm{HT}}_{\alpha}\) denotes \eqref{eqn: CVest} with
\(g=g^{\mathrm{HT}}_{\lambda,\tau}\).

By \eqref{eq:hard-threshold-q-identity},
$\mathbb E\left[
\mathcal V^{\mathrm{HT}}_{\alpha}
\,\middle|\,
Y
\right]
=
0$.
Thus, the adjustment leaves the conditional expectation, and hence the
bias and S-VAR, unchanged.

For direct implementation, define
\begin{equation*}
u_{j,+,\alpha}
=
\frac{\tau_j-z_j}
{\sigma\sqrt\alpha\,\|c_j\|_2},
\qquad
u_{j,-,\alpha}
=
\frac{\tau_j+z_j}
{\sigma\sqrt\alpha\,\|c_j\|_2}.
\end{equation*}
Since \(a_j(Y)=x_jz_j\),
\begin{align*}
\left\langle
x_jz_j,q_{j,\alpha}(Y)
\right\rangle
&=
\frac{\sigma z_jx_j\tran c_j}{\|c_j\|_2}
\left\{
\varphi(u_{j,+,\alpha})
-
\varphi(u_{j,-,\alpha})
\right\},
\\
\left\langle
x_jz_j,\omega^kI_{j,\alpha}^k
\right\rangle
&=
z_j(x_j\tran\omega^k)I_{j,\alpha}^k.
\end{align*}
Consequently,
\eqref{eq:hard-thresholded-ridge-control-variate} can equivalently be
written as
\begin{align*}
\mathcal V^{\mathrm{HT}}_{\alpha}
={}&
\frac{2}{\sqrt\alpha}
\sum_{j\in\mathcal J_{\lambda}}
\frac{\sigma z_jx_j\tran c_j}{\|c_j\|_2}
\left\{
\varphi(u_{j,+,\alpha})
-
\varphi(u_{j,-,\alpha})
\right\}
-
\frac{2}{K\sqrt\alpha}
\sum_{k=1}^K
\sum_{j\in\mathcal J_{\lambda}}
z_j(x_j\tran\omega^k)I_{j,\alpha}^k.
\end{align*}
The calculation uses only the marginal law
\(\omega^k\sim\calN(0,\sigma^2I_n)\) and independence of
\(\omega^k\) from \(Y\); joint normality of
\((\omega^1,\ldots,\omega^K)\) is not required.

\section{Additional numerical results}
\label{app:additional:figures}

\Cref{fig:plot2} reports the MSE under the same simulation setup as the reducible-variance analysis in \Cref{sec: numerical}. We show only the antithetic schemes, since the non-antithetic schemes have substantially larger reducible variance and hence much larger MSE. The rankings mirror those in \Cref{fig:plot1}: ACV (Normal) outperforms ACV (Rademacher) for ridge regression (panel (a)), while the control-variate adjustment further reduces the MSE for hard-thresholded ridge regression (panel (b)). In panel (b), ACV (Normal) and ACV (Adjusted) have the same bias and \(\mathrm{S\text{-}VAR}\), so the difference in their MSE is entirely attributable to \(\mathrm{R\text{-}VAR}\). For ACV (Adjusted), the increase in MSE as \(\alpha\) decreases is driven by the increase in \(\mathrm{S\text{-}VAR}\), since its \(\mathrm{R\text{-}VAR}\) remains bounded and the bias vanishes as $\alpha\downarrow0$.

\begin{figure}[h!]
    \centering
    \includegraphics[width = .9\textwidth]{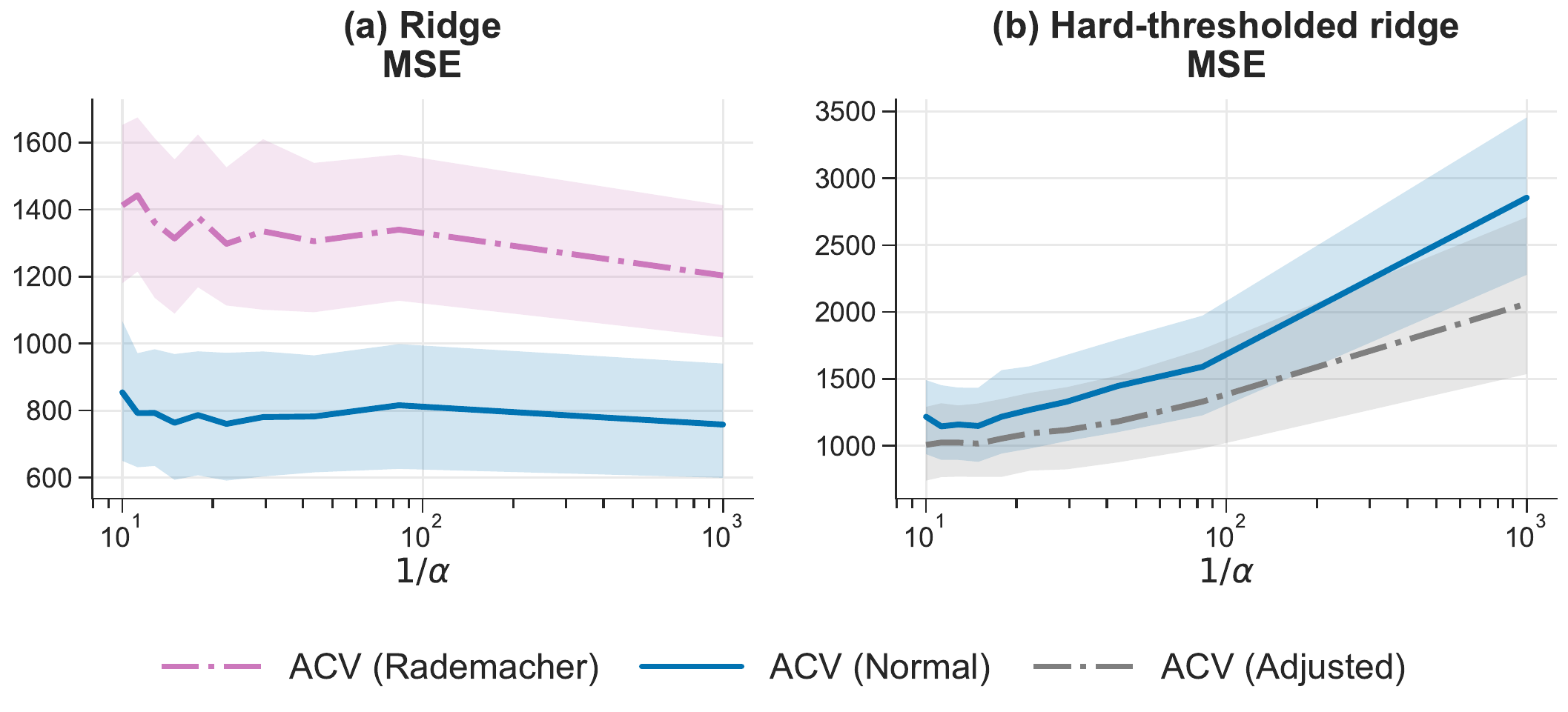}
    \caption{Cross-validation for ridge estimators with and without thresholding. Panel (a) plots the MSE of cross-validation estimators for ridge regression against $1/\alpha$. Panel (b) plots the MSE for coordinatewise hard-thresholded ridge regression against $1/\alpha$.}
    \label{fig:plot2}
\end{figure}

\end{document}